\documentclass[review]{elsarticle}

\usepackage{amsfonts}
\usepackage{amsmath}
\usepackage{lineno}
\usepackage{diagrams}
\usepackage{hyperref}

\modulolinenumbers[5]

\journal{Journal of \LaTeX\ Templates}

\newcommand{\en}{\mathbb N}
\newcommand{\er}{\mathbb R}

\newtheorem{theorem}{Theorem}
\newtheorem{proposition}[theorem]{Proposition}

\newtheorem{definition}[theorem]{Definition}
\newtheorem{question}[theorem]{Question}
\newdefinition{remark}[theorem]{Remark}
\newproof{proof}{Proof}
\newproof{pot}{Proof of Theorem \ref{hlavni veta}}

\begin{document}

\begin{frontmatter}

\title{Classification of the spaces $C_p^*(X)$ within the Borel-Wadge hierarchy for a projective space $X$}


\author[MU]{Martin Dole\v zal\fnref{MD}}
\cortext[mycorrespondingauthor]{Corresponding author}
\fntext[MD]{The author was supported by RVO: 67985840.}
\ead{dolezal@math.cas.cz}

\author[MFF]{Benjamin Vejnar\corref{mycorrespondingauthor}\fnref{BV}}
\fntext[BV]{The author was supported by the grant GA\v CR 14-06989P.
The author is a junior researcher in the University Center for Mathematical Modeling, Applied Analysis and Computational Mathematics (Math MAC).}
\ead{vejnar@karlin.mff.cuni.cz}

\address[MU]{Mathematical Institute, Czech Academy of Sciences,\\
\v Zitn\'a 25, 115 67 Praha 1, Czech Republic}

\address[MFF]{Charles University in Prague, Faculty of Mathematics and Physics,\\
Ke Karlovu 3, 121 16 Praha 2, Czech Republic}

\begin{abstract}

We study the complexity of the space $C^*_p(X)$ of bounded continuous functions with the topology of pointwise convergence. We are allowed to use descriptive set theoretical methods, since for a separable metrizable space $X$, the measurable space of Borel sets in $C^*_p(X)$ (and also in the space $C_p(X)$ of all continuous functions) is known to be isomorphic to a subspace of a standard Borel space. It was proved by A. Andretta and A. Marcone
that if $X$ is a $\sigma$-compact metrizable space, then the measurable spaces $C_p(X)$ and $C^*_p(X)$ are standard Borel and if $X$ is a metrizable analytic space which is not $\sigma$-compact then the spaces of continuous functions are Borel-$\Pi^1_1$-complete. 
They also determined under the assumption of projective determinacy (\textsf{PD}) the complexity of $C_p(X)$ for any projective space $X$ and asked whether a similar result holds for $C^*_p(X)$.

We provide a positive answer, i.e. assuming \textsf{PD} we prove, that if $n \geq 2$ and if $X$ is a separable metrizable space which is in $\Sigma^1_n$ but not in $\Sigma^1_{n-1}$ then the measurable space $C^*_p(X)$ is Borel-$\Pi^1_n$-complete.
This completes under the assumption of \textsf{PD} the classification of Borel-Wadge complexity of $C^*_p(X)$ for $X$ projective.

\end{abstract}

\begin{keyword}

Wadge hierarchy \sep function spaces \sep pointwise convergence \sep projective determinacy

\MSC[2010] 03E15 \sep 28A05 \sep 54C35

\end{keyword}

\end{frontmatter}


\section{Introduction}

First of all, we recall the needed terminology. Most of the definitions in this introductory section are taken from \citep{Kechris} (and \citep{AM}). In Definitions \ref{def1} and \ref{def2}, we recall the Wadge hierarchy.

\begin{definition}
\label{def1}

Let $X$ and $Y$ be topological spaces and let $A$, $B$ be subsets of $X$, $Y$, respectively. We say that $A$ is Wadge reducible to $B$ and write $(A,X) \leq_W (B,Y)$ (or simply $A \leq_W B$ if the spaces $X$ and $Y$ are understood) if there exists a continuous map $f \colon X \rightarrow Y$ (called a Wadge reduction of $A$ to $B$) such that $A = f^{-1}(B)$.

\end{definition}

\begin{definition}
\label{def2}

Let $\Gamma$ be a class of sets in Polish spaces. Let $X$ be a Polish space and $A$ be a subset of $X$. We say that $A$ is $\Gamma$-hard in $X$ if for any zero-dimensional Polish space $Y$ and any $B \in \Gamma(Y)$, we have $(B,Y) \leq_W (A,X)$. If, moreover, $A \in \Gamma(X)$, we say that $A$ is $\Gamma$-complete in $X$.

\end{definition}

In this paper, we are mostly interested in the Borel-Wadge hierarchy, where the notions of topology and continuity are replaced by notions of $\sigma$-algebra and measurability. This is recalled in Definitions \ref{def3}, \ref{sbs} and \ref{def4}.

\begin{definition}
\label{def3}

Let $X$ and $Y$ be measurable spaces and let $A$, $B$ be subsets of $X$, $Y$, respectively. We say that $A$ is Borel-Wadge reducible to $B$ and write $(A,X) \leq_B (B,Y)$ (or simply $A \leq_B B$ if the spaces $X$ and $Y$ are understood) if there exists a measurable map $f \colon X \rightarrow Y$ (called a Borel-Wadge reduction of $A$ to $B$) such that $A = f^{-1}(B)$.

\end{definition}

\begin{definition}
\label{sbs}

A measurable space $(X, \mathcal S)$ is called a standard Borel space if there is a Polish space $(Y, \tau)$ and an isomorphism (of measurable spaces) $f$ of $(X, \mathcal S)$ onto $(Y, \mathcal B(\tau))$ where $\mathcal B(\tau)$ is the Borel $\sigma$-algebra on $Y$ generated by the topology $\tau$.

\end{definition}

\begin{definition}
\label{def4}

Let $\Gamma$ be a class of sets in standard Borel spaces. Let $X$ be a standard Borel space and $A$ be a subset of $X$. We say that $A$ is Borel-$\Gamma$-hard in $X$ if for any standard Borel space $Y$ and any $B \in \Gamma(Y)$, we have $(B,Y) \leq_B (A,X)$. If, moreover, $A \in \Gamma(X)$, we say that $A$ is Borel-$\Gamma$-complete in $X$.

\end{definition}

The projective classes $\Sigma^1_n$ or $\Pi^1_n$, $n \in \en$, (see e.g. \citep[37.A]{Kechris}) are most often considered as classes of sets in Polish spaces but they can be also considered as classes of sets in standard Borel spaces due to Definition \ref{def5}.

\begin{definition}
\label{def5}

Let $\Gamma$ be one of the classes $\Sigma^1_n$ or $\Pi^1_n$, $n \in \en$, of sets in Polish spaces. Let $X$ be a standard Borel space and let $A$ be a subset of $X$. We say that $A \in \Gamma(X)$ if for some Polish space $Y$ and some Borel isomorphism $f \colon X \rightarrow Y$ (or equivalently, for any Polish space $Y$ and any Borel isomorphism $f \colon X \rightarrow Y$), we have $f(A) \in \Gamma (Y)$.

\end{definition}

\begin{remark}
\label{vysvetleni}

Suppose that $\Gamma$ is one of the classes $\Sigma^1_n$ or $\Pi^1_n$, $n \in \en$. Let $(X, \tau)$ be a Polish space and $A$ be a subset of $X$ which is $\Gamma$-hard in $X$ (resp. $\Gamma$-complete in $X$). If we consider $A$ as a subset of the standard Borel space $(X, \mathcal B(X))$ (i.e., of $X$ endowed with the Borel $\sigma$-algebra generated by $\tau$) then $A$ is also Borel-$\Gamma$-hard in $X$ (resp. Borel-$\Gamma$-complete in $X$). This easily follows from the fact that any standard Borel space $(Y, \mathcal S)$ can be endowed with a zero-dimensional Polish topology $\nu$ such that the Borel $\sigma$-algebra generated by $\nu$ equals to $\mathcal S$ (see e.g. \citep[Exercise 13.5]{Kechris}).

\end{remark}

Definition \ref{def6} enables us to describe the projective degree of any separable metrizable topological space.

\begin{definition}
\label{def6}

Let $\Gamma$ be one of the classes $\Sigma^1_n$ or $\Pi^1_n$, $n \in \en$, of sets in Polish spaces and let $X$ be a separable metrizable space. We say that $X$ is in $\Gamma$ if for some Polish space $Y$ and some homeomorphism $f$ of $X$ into $Y$ (or equivalently, for any Polish space $Y$ and any homeomorphism $f$ of $X$ into $Y$), we have $f(X) \in \Gamma (Y)$.

We say that a separable metrizable space is projective if it is in $\Sigma^1_n$ for some $n \in \en$ (or equivalently, if it is in $\Pi^1_n$ for some $n \in \en$).

\end{definition}

Similarly, due to Definition \ref{def7}, we can describe the projective degree of an arbitrary measurable space which can be embedded into a standard Borel space. Unlike the previous definitions, we have not found any explanation of the correctness of this definition so we provide at least a short explanation here. Let $\Gamma$ be one of the classes $\Sigma^1_n$ or $\Pi^1_n$, $n \in \en$, of sets in standard Borel spaces. Suppose that $X$ is a measurable space, $Y$, $Z$ are standard Borel spaces, $f$ is an isomorphism (of measurable spaces) of $X$ into $Y$ and $g$ is an isomorphism (of measurable spaces) of $X$ into $Z$. Then $f(X)$ is in $\Gamma(Y)$ (resp. Borel-$\Gamma$-hard in $Y$ or Borel-$\Gamma$-complete in $Y$) if and only if $g(X)$ is in $\Gamma(Z)$ (resp. Borel-$\Gamma$-hard in $Z$ or Borel-$\Gamma$-complete in $Z$). Indeed, by \citep[Exercise 12.3]{Kechris}, there are Borel sets $A \subseteq Y$, $B \subseteq Z$ with $f(X) \subseteq A$, $g(X) \subseteq B$ and a Borel isomorphism $\Phi \colon A \rightarrow B$ extending the Borel isomorphism $g \circ f^{-1}$ of $f(X)$ and $g(X)$. 
\begin{diagram}[height=1.4em,width=1.em]
$$X$$ & \rTo^f & $$f(X)$$& $$\subseteq$$& $$A$$ & $$\subseteq$$ & $$Y$$\\
&\rdTo_g&&&\dTo>\Phi&&\\
&&$$g(X)$$&$$\subseteq$$ & $$B$$& $$\subseteq$$ & $$Z$$\\
\end{diagram}
From this, the conclusion easily follows and we can formulate the definition.

\begin{definition}
\label{def7}

Let $\Gamma$ be one of the classes $\Sigma^1_n$ or $\Pi^1_n$, $n \in \en$, of sets in standard Borel spaces and let $X$ be a measurable space. We say that $X$ is in $\Gamma$ (resp. Borel-$\Gamma$-hard or Borel-$\Gamma$-complete) if for some standard Borel space $Y$ and some isomorphism (of measurable spaces) $f$ of $X$ into $Y$ (or equivalently, for any standard Borel space $Y$ and any isomorphism (of measurable spaces) $f$ of $X$ into $Y$), the set $f(X)$ is in $\Gamma(Y)$ (resp. Borel-$\Gamma$-hard in $Y$ or Borel-$\Gamma$-complete in $Y$).

\end{definition}

Let $X$ be a separable metrizable topological space. We denote by $C_p(X)$ (resp. $C_p^*(X)$) the measurable space of all real continuous (resp. all real bounded continuous) functions on $X$ equipped with the Borel $\sigma$-algebra generated by the topology of pointwise convergence on $X$. The Borel-Wadge degree of the measurable spaces $C_p(X)$ and $C_p^*(X)$ was already studied e.g. in \citep{Christensen, DM, Ok, AM}, all the relevant results being summarized in \citep{AM}. By these results, it is already known that if $X$ is $\sigma$-compact then both $C_p(X)$ and $C_p^*(X)$ are standard Borel spaces. And if $X$ is in $\Sigma^1_1$ but not $\sigma$-compact then both $C_p(X)$ and $C_p^*(X)$ are Borel-$\Pi^1_1$-complete (see \citep[Corollary 3.4]{AM}). This completely classifies the Borel-Wadge degree of these spaces for $X$ in $\Sigma^1_1$. Now suppose that $X$ is projective but not in $\Sigma^1_1$. Then both $C_p(X)$ and $C_p^*(X)$ are in $\Pi^1_n$ where $n \geq 2$ is the first such that $X$ is in $\Sigma^1_n$ (see \citep[Lemma 2.3 and the last paragraph in Section 2]{AM}). The precise Borel-Wadge degree of $C_p(X)$ in this case was also studied in \citep{AM} under the additional assumption of projective determinacy (henceforth denoted by \textsf{PD}). The principle of \textsf{PD} states that every infinite game $G(\en,X)$ with a projective payoff set $X \subseteq \en^{\en}$ is determined (for more detailed information, we refer to \citep{Kechris}). It was shown in \citep[Theorem 4.3]{AM} that if $n \geq 2$ is the first such that $X$ is in $\Sigma^1_n$ then $C_p(X)$ is Borel-$\Pi^1_n$-complete (under \textsf{PD}). So the Borel-Wadge degree of $C_p(X)$ for $X$ projective is also completely classified (under \textsf{PD}). But the proof of \citep[Theorem 4.3]{AM} uses unbounded functions in an essential way, and so the Borel-Wadge degree of $C_p^*(X)$ for $X$ projective remained unresolved. Instead, the following question was posed in \citep{AM}.

\begin{question}[{\citep[Problem 4.4]{AM}}]

Assume \emph{\textsf{PD}}. Let $X$ be a separable metrizable projective space which is not in $\Sigma^1_1$. Let $n \geq 2$ be the first such that $X$ is in $\Sigma^1_n$. Is the measurable space $C_p^*(X)$ Borel-$\Pi^1_n$-complete?

\end{question}

In this paper, we positively answer this question by proving the following main theorem.

\begin{theorem}
\label{hlavni veta}

Assume \emph{\textsf{PD}}. Let $X$ be a separable metrizable projective space which is not in $\Sigma^1_1$. Let $n \geq 2$ be the first such that $X$ is in $\Sigma^1_n$. Then the measurable space $C^*_p(X)$ is Borel-$\Pi^1_n$-complete.

\end{theorem}

Under \textsf{PD}, this concludes the classification of the Borel-Wadge degree of $C_p^*(X)$ for $X$ projective. As in \citep{AM}, we use \textsf{PD} to know that if $n \in \en$, $X$ is a Polish space and $A$ is a subset of $X$ which is not in $\Pi^1_n$ then $A$ is $\Sigma^1_n$-hard in $X$. If $X$ is zero-dimensional, this follows (under \textsf{PD}) from an easy analogy of \citep[Theorem 21.14 (Wadge's Lemma)]{Kechris}. The general case can be reduced to the previous one by \citep[Exercise 13.5]{Kechris}.

One of the main ideas of the proof of Theorem \ref{hlavni veta} is the same as in \citep{AM}, i.e. providing a Borel-Wadge reduction of the $\Pi^1_n$-hard subset $\mathcal K(W \setminus X)$ of $\mathcal K(W)$ (where $W = Z \setminus D$, $D$ is a countable dense subset of $X$ and $Z$ is a metric completion of $X$) to $C^*_p(X)$. This is done almost in the same way as in \citep[proof of Theorem 4.3]{AM} in the particular case of $X$ being nowhere locally compact. The only refinement is hidden in the fact that in this case, the completion $Z$ of $X$ can be chosen to be a Peano continuum due to \citep[Corollary 7]{Dijkstra}. The general case is then reduced to this particular one by using Proposition \ref{redukce} which seems to be very intuitive but not trivial.

\section{Proof of the main theorem}

By the well known Tietze extension theorem, every real continuous function on a closed subspace $H$ of a metric space $X$ can be extended to a real continuous function on $X$. We will need the following version of this theorem since it provides a simple formula for the extension. This version of the Tietze theorem is due to F.~Riesz and was published in Ker\'ekj\'art\'o's book \citep{Ke} but probably the most accessible source is \citep{Husek}.

\begin{theorem}
\label{Tietze}

Let $(X,d)$ be a metric space and $H$ be a closed subset of $X$. Let $f \colon H \rightarrow [1,2]$ be continuous. Then the function $F \colon X \rightarrow \er$ defined by

$$ F(x) =
\begin{cases}
f(x) & \text{if } x \in H \\
\inf \left\{ f(h) \frac {d(x,h)} {d(x,H)} \colon h \in H \right\} & \text{if }x \in X \setminus H
\end{cases}
$$
is continuous with values in the interval $[1,2]$.

\begin{remark}
\label{remark}

Assume the hypothesis and notation of Theorem \ref{Tietze}. Let $D$ be an arbitrary dense subset of $H$. Then for every $x \in X \setminus H$, we clearly have
$$ F(x) = \inf \left\{ f(a) \frac {d(x,a)} {d(x,H)} \colon a \in D \right\}. $$

\end{remark}

\end{theorem}

For a separable metrizable space $X$ and a countable dense subset $D$ of $X$, we put
$$ \tilde C_p^*(X,D) = \{(r_a)_{a \in D} \in \er^D \colon \exists f \in C_p^*(X) \ \forall a \in D \ f(a) = r_a \}. $$
Let $\phi \colon C_p^*(X) \rightarrow \er^D$ be defined by $\phi(f) = (f(a))_{a \in D}$, $f \in C_p^*(X) $. It was shown in \citep[Lemma 2.2 and the last paragraph of Section 2]{AM} that $\phi$  is an isomorphism of the measurable space $C_p^*(X)$ and the measurable subspace $\tilde C_p^*(X,D)$ of the standard Borel space $\er^D$ (which is equipped with the Borel $\sigma$-algebra generated by the product topology). So we only need to examine the Borel-Wadge degree of the set $\tilde C_p^*(X,D)$ in $\er^D$. Observe also that the countable dense subset $D$ of $X$ can be chosen arbitrarily.

\begin{proposition}
\label{redukce}

Let $X$ be a separable metrizable space and $H$ be a closed subspace of $X$. Let $D$, $E$ be countable dense subsets of $H$, $X \setminus H$ respectively (so that $D \cup E$ is countable dense in $X$). Then $(\tilde C^*_p(H, D), \er^D) \leq_B (\tilde C^*_p(X, D \cup E), \er^{D \cup E})$.

\end{proposition}

\begin{proof}

We may suppose that the cardinality of $X$ is infinite since the other case is trivial. In two steps, we will show that
$$ (\tilde C^*_p(H, D), \er^D) \leq_B (\tilde C_p^*(H, D) \cap [1,2]^D, \er^D) \leq_B (\tilde C^*_p(X, D \cup E), \er^{D \cup E}). $$

(i) $(\tilde C^*_p(H, D), \er^D) \leq_B (\tilde C_p^*(H, D) \cap [1,2]^D, \er^D)$

\noindent For $n \in \en$, let $\psi_n$ be a homeomorphism of the interval $[-n, n]$ onto $[1,2]$. We define $\phi \colon \er^D \rightarrow \er^D$ by
$$ \phi((r_a)_{a \in D}) =
\begin{cases}
(\psi_n(r_a))_{a \in D} & \text{if } (r_a)_{a \in D} \in [-n,n]^D \\
& \text{and } n \in \en \text{ is the least with this property}, \\
(r_a)_{a \in D} & \text{if } (r_a)_{a \in D} \in \er^D \setminus \bigcup\limits_{n \in \en} [-n,n]^D.
\end{cases}
$$
Then $\phi$ is obviously Borel measurable. We will show that $\phi$ is the required Borel-Wadge reduction of $\tilde C^*_p(H, D)$ to $\tilde C_p^*(H, D) \cap [1,2]^D$.

Suppose that $(r_a)_{a \in D} \in \tilde C^*_p(H, D)$ and let $n \in \en$ be the first such that $(r_a)_{a \in D}$ is bounded by $n$. Then there is $f \in C_p^*(X)$ with values in the interval $[-n, n]$ which extends $(r_a)_{a \in D}$ from $D$ to $X$. But then $\psi_n \circ f$ is a continuous function with values in the interval $[1, 2]$ extending $\phi((r_a)_{a \in D}) = (\psi_n(r_a))_{a \in D}$ from $D$ to $X$, and so $\phi((r_a)_{a \in D}) \in \tilde C_p^*(H, D) \cap [1,2]^D$.

Now suppose that $(r_a)_{a \in D} \in \er^D \setminus \tilde C^*_p(H, D)$. If $(r_a)_{a \in D}$ is bounded then it has no continuous extension from $D$ to $X$. Let $n \in \en$ be the first such that $(r_a)_{a \in D}$ is bounded by $n$, then neither $\phi((r_a)_{a \in D}) = (\psi_n(r_a))_{a \in D}$ has a continuous extension from $D$ to $X$, and so $\phi((r_a)_{a \in D}) \in \er^D \setminus \tilde C_p^*(H, D) \subseteq \er^D \setminus \left( \tilde C_p^*(H, D) \cap [1,2]^D \right)$. And if $(r_a)_{a \in D}$ is unbounded then $\phi((r_a)_{a \in D}) = (r_a)_{a \in D}$ is unbounded, too. So again, $\phi((r_a)_{a \in D}) \in \er^D \setminus \left( \tilde C_p^*(H, D) \cap [1,2]^D \right)$.

(ii) $(\tilde C_p^*(H, D) \cap [1,2]^D, \er^D) \leq_B (\tilde C^*_p(X, D \cup E), \er^{D \cup E})$

\noindent Let $d$ be a compatible metric on $X$. We fix arbitrary unbounded $( s_b )_{b \in D \cup E} \in \er^{D \cup E}$ (this is possible since $D \cup E$ is clearly infinite) and define $\phi \colon \er^D \rightarrow \er^{D \cup E}$ by
$$ \phi ((r_a)_{a \in D}) (b) =
\begin{cases}
s_b & \text{if } (r_a)_{a \in D} \notin [1,2]^D \text{ (and } b \in D \cup E \text{)}, \\
r_b & \text{if } (r_a)_{a \in D} \in [1,2]^D \text{ and } b \in D, \\
\inf \left\{ r_a \frac {d(b,a)} {d(b,H)} \colon a \in D \right\} & \text{if } (r_a)_{a \in D} \in [1,2]^D \text{ and } b \in E.
\end{cases}
$$
We will show that $\phi$ is the required Borel-Wadge reduction of $\tilde C_p^*(H, D) \cap [1,2]^D$ to $\tilde C_p^*(X, D \cup E)$.

For fixed $a_0 \in D$ and $b_0 \in E$, the function $(r_a)_{a \in D} \mapsto r_{a_0} \frac {d(b_0,a_0)} {d(b_0,H)} \in \er$ is continuous (since $\frac {d(b_0,a_0)} {d(b_0,H)}$ is constant). So for every $b_0 \in E$, the function $(r_a)_{a \in D} \mapsto \inf \left\{ r_a \frac {d(b_0,a)} {d(b_0,H)} \colon a \in D \right\} \in \er$ is upper semicontinuous (since it is an infimum of continuous functions). It immediately follows that $\phi$ is Borel measurable.

Suppose that $(r_a)_{a \in D} \in \er^D \setminus \left( \tilde C_p^*(H, D) \cap [1,2]^D \right)$. If $(r_a)_{a \in D} \notin [1,2]^D$ then $\phi((r_a)_{a \in D})$ is unbounded and so $\phi((r_a)_{a \in D}) \in \er^{D \cup E} \setminus \tilde C_p^*(X, D \cup E) $. Otherwise, $(r_a)_{a \in D} \in [1, 2]^D \setminus \tilde C_p^*(H, D)$. Then $(r_a)_{a \in D}$ has no bounded continuous extension from $D$ to $H$ and so $\phi((r_a)_{a \in D})$ clearly has no bounded continuous extension from $D \cup E$ to $X$, in other words again $\phi((r_a)_{a \in D}) \in \er^{D \cup E} \setminus \tilde C_p^*(X, D \cup E)$.

Now suppose that $(r_a)_{a \in D} \in \tilde C_p^*(H, D) \cap [1,2]^D$. Then there is $f \in C_p^*(X)$ with values in $[1,2]$ such that $f(a) = r_a$ for every $a \in D$. By Theorem \ref{Tietze} (used on such $f$) and Remark \ref{remark}, there is $F \in C_p^*(X)$ with values in $[1,2]$ such that $F(b) = \phi((r_a)_{a \in D})(b)$ for every $b \in D \cup E$, and so $\phi((r_a)_{a \in D}) \in \tilde C_p^*(X, D \cup E)$.
\qed

\end{proof}

In the following, by a \emph{continuum}, we mean a nonempty, compact, connected, metrizable topological space. A \emph{Peano continuum} is a locally connected continuum. A topological space is \emph{nowhere locally compact} if no point has a compact neighborhood. One of the tools we will need is the following theorem from \citep{Dijkstra} (nowhere locally compact spaces are called just nowhere compact in \citep{Dijkstra}).

\begin{theorem}[{\citep[Corollary 7]{Dijkstra}}]
\label{compactification}

Every separable metrizable nowhere locally compact space has a compactification which is a Peano continuum.

\end{theorem}

Now we are ready for the proof of Theorem \ref{hlavni veta}.

\begin{pot}

By \citep[Lemma 2.3 and the last paragraph in Section 2]{AM}, the measurable space $C_p^*(X)$ is in $\Pi^1_n$ so we only need to show that it is Borel-$\Pi^1_n$-hard. 

Suppose first that $X$ is nowhere locally compact. Then by Theorem \ref{compactification}, $X$ has a compactification $Z$ which is a Peano continuum. Let $d$ be a compatible metric on $Z$ and let $D$ be a countable dense subset of $X$ (then $D$ is also dense in $Z$). We put $W = Z \setminus D$ (so that $W$ is a Polish space when equipped with the topology inherited from $Z$) and $Y = Z \setminus X$. Since $X$ is not $\Sigma^1_{n-1}$ in $Z$, neither is $X \setminus D$ (because $D$ is countable and thus $\Sigma^0_2$ in $Z$). So $X \setminus D$ is not $\Sigma^1_{n-1}$ neither in $W$. It follows that $Y = W \setminus (X \setminus D)$ is not $\Pi^1_{n-1}$ in $W$. By \textsf{PD}, $Y$ is $\Sigma^1_{n-1}$-hard in $W$. By this and \citep[Lemma 4.2]{AM} (in fact, this lemma was first proved in \citep[Lemma 1]{KLW} but formulated only under some more restrictive assumptions on the space $W$), it follows that $\mathcal K(Y)$ (i.e., the set of all compact subsets of $Y$) is $\Pi^1_n$-hard in $\mathcal K(W)$ (i.e., the Polish space of all compact subsets of $W$ equipped with the Vietoris topology). So, due to Remark \ref{vysvetleni}, it suffices to show that $(\mathcal K(Y),\mathcal K(W)) \leq_W (\tilde C^*_p(X,D),\er^D)$.

We define $\phi \colon \mathcal K(W) \rightarrow \er^D$ by
$$ \phi(K) = \left( \sin \left( \frac 1 {d(a,K)} \right) \right)_{a \in D}, \ \ \ K \in \mathcal K(W). $$
This is a correct definition since for every $K \in \mathcal K(W)$, we have $K \cap D = \emptyset$ and so $d(a,K) > 0$ for every $a \in D$. The map $\phi$ is obviously continuous. We will show that $\phi$ is the required Wadge reduction of $\mathcal K(Y)$ to $\tilde C^*_p(X,D)$. If $K \in \mathcal K(Y)$ then $K \cap X = \emptyset$ and so the function $f \colon X \rightarrow \er$ defined by $f(x) = \sin \left(\frac 1 {d(x,K)}\right)$, $x \in X$, is a bounded continuous extension of $\phi(K)$ from $D$ to $X$. So we have $\phi(K) \in \tilde C^*_p(X,D)$ whenever $K \in \mathcal K(Y)$. Now suppose that $K \in \mathcal K(W) \setminus \mathcal K(Y)$. Then we can find some $x \in K \cap X$. Since $Z$ is a Peano continuum, it is locally connected at $x$. So for every $n \in \en$, there is an open connected subset $U_n$ of $Z$ such that $x \in U_n \subseteq \{ z \in Z \colon d(z, x) < \frac 1 n \}$. Let us fix $n \in \en$ for now. The function $\psi \colon Z \rightarrow \er$ defined by $\psi (z) = d(z, K)$, $z \in Z$, is continuous, and so the image $\psi (U_n)$ of the set $U_n$ is a connected subset of $\er$. Since $x \in K \cap U_n$, we have $0 \in \psi (U_n)$. And since $D$ intersects $U_n$ (it is dense in $Z$) and $D \cap K = \emptyset$, the function $\psi$ also attains some positive values in $U_n$. So there is $k_n \in \en$ such that the connected set $\psi (U_n)$ contains the interval $[0, \frac 1 { 2k_n\pi })$. Then we can find $z_n \in U_n$ such that
$$ \psi (z_n) =
\begin{cases}
\frac 1 { 2k_n \pi + \frac 1 2 {\pi}} & \text{if } n \text{ is odd}, \\
\frac 1 { 2k_n \pi + \frac 3 2 \pi } & \text{if } n \text{ is even}.
\end{cases}
$$
In this way, we obtain a sequence $(z_n)_{n \in \en}$ in $Z$ such that for every $n \in \en$, we have $z_n \in U_n$, $\psi (z_n) > 0$ and
$$ \sin \left( \frac {1} {\psi (z_n)} \right) =
\begin{cases}
1 & \text{if } n \text{ is odd}, \\
-1 & \text{if } n \text{ is even}.
\end{cases}
$$
The function $z \in Z \setminus K \mapsto \sin \left( \frac 1 { \psi (z) } \right)$ is continuous and $D \cup \{ z_n \colon n \in \en \}$ is a subset of its domain $Z \setminus K$, so by the density of $D$ in $Z$, we can find a sequence $(a_n)_{n \in \en}$ in $D$ such that for every $n \in \en$, we have $a_n \in U_n$ and
$$ \phi (K) (a_n) = \sin \left( \frac {1} {\psi (a_n)} \right)
\begin{cases}
>  \frac 1 2 & \text{if } n \text{ is odd}, \\
< - \frac 1 2 & \text{if } n \text{ is even}.
\end{cases}
$$
Clearly $\lim\limits_{n \rightarrow \infty} a_n = x$ but $\lim\limits_{n \rightarrow \infty} \phi (K) (a_n)$ does not exist. So $\phi(K)$ cannot be continuously extended from $D$ to $D \cup \{ x \}$. It follows that $\phi(K) \in \er^D \setminus \tilde C^*_p(X,D)$ whenever $K \in \mathcal K(W) \setminus \mathcal K(Y)$. This completes the proof for $X$ being nowhere locally compact.

Now suppose that $X$ is arbitrary. We will construct a decreasing (with respect to inclusion) transfinite sequence $(X_{\alpha})_{\alpha < \omega_1}$ of closed subspaces of $X$ in the following way. We start by $X_0=X$. Now suppose that for some $\alpha < \omega_1$, we already have $X_{\beta}$ for every $\beta < \alpha$. If $\alpha = \beta + 1$ for some ordinal $\beta$, we put
$$ X_{\alpha} = X_{\beta} \setminus \bigcup \{ V \subseteq X_{\beta} \colon V \text{ is open and relatively compact in } X_{\beta} \}. $$
And if $\alpha$ is a limit ordinal, we put $X_{\alpha} = \bigcap\limits_{\beta < \alpha} X_{\beta}$. Since $X$ is second countable, there is some $\alpha < \omega_1$ such that $X_{\alpha + 1} = X_{\alpha}$. Let $\alpha_0$ be the least such $\alpha$, then $X_{\alpha_0}$ is clearly nowhere locally compact. Moreover, it easily follows by the hereditary Lindel\"ofness of $X$ that $X \setminus X_{\alpha_0}$ is contained in a $\sigma$-compact subset of $X$, and so the topological space $X_{\alpha_0}$ is clearly in $\Sigma^1_n$ but not in $\Sigma^1_{n-1}$. By the previous step, the measurable space $C^*_p(X_{\alpha_0})$ is Borel-$\Pi^1_n$-complete. Let $D$, $E$ be countable dense subsets of $X_{\alpha_0}$, $X \setminus X_{\alpha_0}$ respectively. Since $X_{\alpha_0}$ is a closed subspace of $X$, we have $(\tilde C^*_p(X_{\alpha_0}, D), \er^D) \leq_B (\tilde C^*_p(X, D \cup E), \er^{D \cup E})$ by Proposition \ref{redukce}. But as it was already explained, the measurable spaces $C_p^*(X_{\alpha_0})$, $C_p^*(X)$ are isomorphic to the subspaces $\tilde C^*_p(X_{\alpha_0}, D)$, $\tilde C^*_p(X, D \cup E)$ of the standard Borel spaces $\er^D$, $\er^{D \cup E}$ respectively, and so the conclusion immediately follows.
\qed

\end{pot}


\bibliography{bibliography}

\end{document}